\documentclass[12pt]{article}
\usepackage{amsmath, amsthm}

\newtheoremstyle{theorem}
  {10pt}		  
  {10pt}  
  {\sl}  
  {\parindent}     
  {\bf}  
  {. }    
  { }    
  {}     
\theoremstyle{theorem}
\newtheorem{theorem}{Theorem}

\newtheoremstyle{defi}
  {10pt}		  
  {10pt}  
  {\rm}  
  {\parindent}     
  {\bf}  
  {. }    
  { }    
  {}     
\theoremstyle{defi}

\begin{document}

\title{A Note on the Hypercomplex Riemann-Cauchy Like Relations for Quaternions and Laplace Equations.}

\author{J. A. P.F. Mar\~ao$^1$ and M. F. Borges$^2$\\
$^1$Department of Mathematics\\ Federal University of Maranh\~ao - S\~ao Lu\'is-MA\\
65085-580, Maranh\~ao, BRAZIL\\
josemarao7@hotmail.com\\[2pt]
$^2$UNESP - S\~ao Paulo State University\\
S.J. Rio Preto Campus\\
15054-000, S\~ao Jos\'e do Rio Preto, BRAZIL\\
borges@ibilce.unesp,br}

\maketitle

\begin{abstract}
In this note it is worked out a new set of Laplace-Like equations for quaternions through Riemann-Cauchy hypercomplex relations otained earlier \cite{BorgesZeMarcio}. As in the theory of functions of a complex variable, it is expected that this new set of Laplace-Like equations might be applied to a large number of Physical problems, providing new insights in the Classical Theory Fields.

{\bf AMS Subject Classification:} 30G99, 30E99

{\bf Key Words and Phrases:}Quaternions, Laplace's Equations, Riemann-Cauchy Relations
\end{abstract}

\section{Cauchy-Riemann Equations (Functions of one complex variable).}
In order to fix ideas will be considered theorem that relates the partial derivatives for the case of a function $f(z)$ of a complex variable $f(z)=u(x,y)+iv(x,y)$ \cite{Kodaira}, which here will be called simply Riemann-Cauchy conditions. These relations say that the first order partial derivatives of functions $u(x,y)$ and $v(x,y)$ satisfy relations according to the following theorem:
\begin{theorem} Is $f(z)=u(x,y)+iv(x,y)$ a function defined and continues in a neighborhood of the point $ z=x+yi$ and differentiable at $z$. Then the partial derivatives of the first order of $u(x,y)$ and $v(x,y)$ exist and satisfy the relations:
\begin{equation}\frac{\partial u(x,y)}{\partial x}=\frac{\partial v(x,y)}{\partial y}\end{equation}
\begin{equation}\frac{\partial u(x,y)}{\partial y}=-\frac{\partial v(x,y)}{\partial x}\end{equation}
\end{theorem}
Thus, if $f(z)$ is analytic in a domain $\Gamma,$ its partial derivatives exist and satisfy the set of relations $(1)$ and $(2)$ over all point in $Gamma.$
Moreover, with the above functions class $C ^{2}$ using Schwartz's Theorem for partial derivatives immediately follows the following equations:
\begin{equation}\frac{\partial^{2}u}{\partial x^{2}}+\frac{\partial^{2}u}{\partial y^{2}}=0\end{equation}
\begin{equation}\frac{\partial^{2}v}{\partial x^{2}}+\frac{\partial^{2}v}{\partial y^{2}}=0\end{equation}
the above equations are called Laplace's equations.

\section{Cauchy-Riemann conditions for quaternionic functions.}
Now the conditions are presented as the Riemann-Cauchy like relations for quaternionic functions. It follows the Theorem \cite{BorgesZeMarcio}:

\begin{theorem}
For any pair pontis $a$ and $b$ and any path joining them simply conect subdomain of the four-dimmensional space, the integral $\int_{a}^{b}fdq$ is independent form the given path if and only if there is a function $F=F_{1}+F_{2}i+F_{3}j+F_{4}k$ such that $\int_{a}^{b}fdq=F(a)-F(b),$ and satisfying the following relations:
\begin{equation}
\frac{\partial F_{1}}{\partial x_{1}}=\frac{\partial F_{2}}{\partial x_{2}}=\frac{\partial F_{3}}{\partial x_{3}}=\frac{\partial F_{4}}{\partial x_{4}}\end{equation}
\begin{equation}\frac{\partial F_{2}}{\partial x_{1}}=-\frac{\partial F_{1}}{\partial x_{2}}=-\frac{\partial F_{3}}{\partial x_{4}}=\frac{\partial F_{4}}{\partial x_{3}}\end{equation}
\begin{equation}\frac{\partial F_{3}}{\partial x_{1}}=-\frac{\partial F_{1}}{\partial x_{3}}=-\frac{\partial F_{2}}{\partial x_{4}}=\frac{\partial F_{4}}{\partial x_{2}}\end{equation}
\begin{equation}\frac{\partial F_{4}}{\partial x_{1}}=\frac{\partial F_{1}}{\partial x_{4}}=-\frac{\partial F_{2}}{\partial x_{3}}=-\frac{\partial F_{3}}{\partial x_{2}}\end{equation}
\end{theorem}

\begin{proof}
The proof of this theorem can be analyzed in greater detail in \cite{BorgesZeMarcio}.
\end{proof}

\section{The Laplace's Equations.}
In this section we show that a new set of hypercomplex Laplace equations may be generated in four dimensions. Through its use of Riemann-Cauchy like relations \cite{BorgesZeMarcio}
Therefore, the functions that make up the quaternionic function, depend on $x_1$,$x_2$,$x_3$ and $x_4$ and are supposed to be of class $C^{2}$ and thus the theorem is valid Schwartz.

The first step to obtain the Laplace equation is the derivation of equations $(5),$ $(6),$ $(7)$ and $(8)$ over $x_{1},$ $x_{2}$, $x_{3}$ and $x_{4}$ will be done as follows:
Deriving the conditions of equation $(5),$ we have:

\begin{equation}
\begin{array}{ccccccccccccccccccccccccccccccccccccccccccccccccccccccccccccccc}
\displaystyle \frac{\partial^{2}F_{1}}{\partial x_{1}^{2}}=\frac{\partial^{2}F_{2}}{\partial x_{1}\partial x_{2}}=\frac{\partial^{2}F_{3}}{\partial x_{1}\partial x_{3}}=\frac{\partial^{2}F_{4}}{\partial x_{1}\partial x_{4}}\\
\displaystyle \frac{\partial^{2}F_{1}}{\partial x_{1}\partial x_{2}}=\frac{\partial^{2}F_{2}}{\partial
x_{2}^{2}}=\frac{\partial^{2}F_{3}}{\partial x_{2}\partial x_{3}}=\frac{\partial^{2}F_{4}}{\partial x_{2}\partial x_{4}}\\
\displaystyle \frac{\partial^{2}F_{1}}{\partial x_{3}\partial x_{1}}=\frac{\partial^{2}F_{2}}{\partial x_{3}\partial x_{2}}=\frac{\partial^{2}F_{3}}{\partial x_{3}^{2}}=\frac{\partial^{2}F_{4}}{\partial x_{4}\partial x_{3}}\\
\displaystyle \frac{\partial^{2}F_{1}}{\partial x_{1}\partial x_{4}}=\frac{\partial^{2}F_{2}}{\partial x_{4}\partial x_{2}}=\frac{\partial^{2}F_{3}}{\partial x_{4}\partial x_{3}}=\frac{\partial^{2}F_{4}}{\partial x_{4}^{2}}.
\end{array}
\end{equation}

Deriving the conditions of equation $(6),$ we have:

\begin{equation}
\begin{array}{ccccccccccccccccccccccccccccccccccccccccccccccccccccccccccccccc}
\displaystyle \frac{\partial^{2}F_{2}}{\partial x_{1}^{2}}=-\frac{\partial^{2}F_{1}}{\partial x_{1}\partial x_{2}}=-\frac{\partial^{2}F_{3}}{\partial x_{1}\partial x_{4}}=\frac{\partial^{2}F_{4}}{\partial x_{1}\partial x_{3}}\\
\displaystyle \frac{\partial^{2}F_{2}}{\partial x_{1}\partial x_{2}}=-\frac{\partial^{1}F_{1}}{\partial
x_{2}^{2}}=-\frac{\partial^{2}F_{3}}{\partial x_{2}\partial x_{4}}=\frac{\partial^{2}F_{4}}{\partial x_{3}\partial x_{2}}\\
\displaystyle \frac{\partial^{2}F_{2}}{\partial x_{3}\partial x_{1}}=-\frac{\partial^{2}F_{1}}{\partial x_{3}\partial x_{2}}=-\frac{\partial^{2}F_{3}}{\partial x_{3} \partial x_{4}}=\frac{\partial^{2}F_{4}}{\partial x_{3}^{2}}\\
\displaystyle \frac{\partial^{2}F_{2}}{\partial x_{4}\partial x_{1}}=-\frac{\partial^{2}F_{1}}{\partial x_{4}\partial x_{2}}=-\frac{\partial0^{2}F_{3}}{\partial x_{4}^{2}}=\frac{\partial^{2}F_{4}}{\partial x_{4}\partial x_{3}}.
\end{array}
\end{equation}

Deriving the conditions of equation $(7),$ we obtain that:

\begin{equation}
\begin{array}{ccccccccccccccccccccccccccccccccccccccccccccccccccccccccccccccc}
\displaystyle \frac{\partial^{2}F_{3}}{\partial x_{1}^{2}}=-\frac{\partial^{2}F_{1}}{\partial x_{1}\partial x_{3}}=-\frac{\partial^{2}F_{2}}{\partial x_{1}\partial x_{4}}=\frac{\partial^{2}F_{4}}{\partial x_{1}\partial x_{2}}\\
\displaystyle \frac{\partial^{2}F_{3}}{\partial x_{1}\partial x_{2}}=-\frac{\partial^{1}F_{1}}{\partial
x_{2}\partial x_{3}}=-\frac{\partial^{2}F_{2}}{\partial x_{2}\partial x_{4}}=\frac{\partial^{2}F_{4}}{\partial x_{2}^{2}}\\
\displaystyle \frac{\partial^{2}F_{3}}{\partial x_{3}\partial x_{1}}=-\frac{\partial^{2}F_{1}}{\partial x_{3}^{2}}=-\frac{\partial^{2}F_{2}}{\partial x_{4}\partial x_{3}}=\frac{\partial^{2}F_{4}}{\partial x_{3}\partial x_{2}}\\
\displaystyle \frac{\partial^{2}F_{3}}{\partial x_{1}\partial x_{4}}=-\frac{\partial^{2}F_{1}}{\partial x_{4}\partial x_{3}}=-\frac{\partial^{2}F_{2}}{\partial x_{4}^{2}}=\frac{\partial^{2}F_{4}}{\partial x_{4}\partial x_{2}}.
\end{array}
\end{equation}

Deriving the conditions of equation $(8),$ we have:

\begin{equation}
\begin{array}{ccccccccccccccccccccccccccccccccccccccccccccccccccccccccccccccc}
\displaystyle \frac{\partial^{2}F_{4}}{\partial x_{1}^{2}}=\frac{\partial^{2}F_{1}}{\partial x_{1}\partial x_{4}}=-\frac{\partial^{2}F_{2}}{\partial x_{1}\partial x_{3}}=-\frac{\partial^{2}F_{3}}{\partial x_{1}\partial x_{2}}\\
\displaystyle\frac{\partial^{2}F_{4}}{\partial x_{1}\partial x_{2}}=\frac{\partial^{1}F_{1}}{\partial
x_{2}\partial x_{4}}=-\frac{\partial^{2}F_{2}}{\partial x_{2}\partial x_{3}}=-\frac{\partial^{2}F_{3}}{\partial x_{2}^{2}}\\
\displaystyle\frac{\partial^{2}F_{4}}{\partial x_{3}\partial x_{1}}=\frac{\partial^{2}F_{1}}{\partial x_{3}\partial x_{4}}=-\frac{\partial^{2}F_{2}}{\partial x_{3}^{2}}=-\frac{\partial^{2}F_{3}}{\partial x_{3}\partial x_{2}}\\
\displaystyle\frac{\partial^{2}F_{4}}{\partial x_{1}\partial x_{4}}=\frac{\partial^{2}F_{1}}{\partial x_{4}^{2}}=-\frac{\partial^{2}F_{2}}{\partial x_{4}\partial x_{3}}=-\frac{\partial^{2}F_{3}}{\partial x_{4}\partial x_{2}}.
\end{array}
\end{equation}

In correlating those derived groups of partial derivatives $(9),$ $(10),$  $(11)$ and $(12),$ immediately following the Laplace Equations:

\begin{equation}\frac{\partial F_{1}}{\partial x_{1}^{2}}+\frac{\partial F_{1}}{\partial x_{2}^{2}}+\frac{\partial F_{1}}{\partial x_{3}^{2}}+\frac{\partial F_{1}}{\partial x_{4}^{2}}=0\end{equation}

\begin{equation}\frac{\partial F_{2}}{\partial x_{1}^{2}}+\frac{\partial F_{2}}{\partial x_{2}^{2}}+\frac{\partial F_{2}}{\partial x_{3}^{2}}+\frac{\partial F_{2}}{\partial x_{4}^{2}}=0\end{equation}

\begin{equation}\frac{\partial F_{3}}{\partial x_{1}^{2}}+\frac{\partial F_{3}}{\partial x_{2}^{2}}+\frac{\partial F_{3}}{\partial x_{3}^{2}}+\frac{\partial F_{3}}{\partial x_{4}^{2}}=0\end{equation}

and

\begin{equation}\frac{\partial F_{4}}{\partial x_{1}^{2}}+\frac{\partial F_{4}}{\partial x_{2}^{2}}+\frac{\partial F_{4}}{\partial x_{3}^{2}}+\frac{\partial F_{4}}{\partial x_{4}^{2}}=0\end{equation}

Therefore, it is more simplified manner, the set of equations appears below:
\begin{equation}
\begin{array}{c}
\Delta f_{1}=0\\
\Delta f_{2}=0\\
\Delta f_{3}=0\\
\Delta f_{4}=0
\end{array}
\end{equation}

\section{Conclusion}
In this not it is showed the feasibility of obtaining the equations of Laplace through the Cauchy-Riemann conditions for quaternions. This fact will allow the relationship between equations that can explain many physical phenomena. You can also use the above equations as a way of stating a theorem for harmonic functions for quaternions that satisfy the conditions of Cauchy results. It is worth mentioning the importance of $\cite{BorgesZeMarcio}$ established relations used in this work.

\end{document}